\documentclass[12pt,reqno]{amsart}
\usepackage{geometry}
\usepackage{amssymb}
\usepackage{amsmath}
\usepackage{amsfonts}
\usepackage{amsthm}
\usepackage{mathrsfs}
\usepackage{graphicx}
\usepackage[table]{xcolor}
\usepackage{xcolor}
\usepackage{enumitem}
\usepackage{url}
\usepackage{tabularx}
\usepackage{hyperref}
\setenumerate{label={$($\normalfont\arabic*)}}

\DeclareMathOperator{\M}{M}
\DeclareMathOperator{\GL}{GL}

\DeclareMathOperator{\Gal}{Gal}

\DeclareMathOperator{\End}{End}

\DeclareMathOperator{\cond}{cond}

\DeclareMathOperator{\Div}{Div}
\def\id#1{{\mathfrak{#1}}}      

\def\g{\delta}
\def\piJLg{\pi_{\JLg}}
\def\pig{\pi_g}
\def\JLg{g_B}
\def \QQ{\mathbb{Q}}
\def \AA{\mathbb{A}}
\def \NN{\mathbb{N}}

\def \ZZ{\mathbb{Z}}
\def \Z{\mathcal{Z}}
\def \z{\mathfrak{z}}

\def \RR{\mathbb{R}}
\def \L{\mathscr{L}}
\def\On{\mathcal{O}}

\def\O{\mathcal{O}}
\def \R{\mathscr{R}}
\def \U{\mathscr{U}}

\def \CC{\mathbb{C}}
\def\<#1>{{\left\langle{#1}\right\rangle}}

\def \tchi{\widetilde{\chi}}

\def \tchi{\widetilde{\chi}}

\theoremstyle{plain}
\newtheorem{thm}{Theorem}[section]
\newtheorem{prop}[thm]{Proposition}

\newtheorem{lemma}[thm]{Lemma}

\theoremstyle{remark}
\newtheorem{remark}[thm]{Remark}

\newtheorem*{rem*}{Remark}
\theoremstyle{definition}
\newtheorem{obs/res}[thm]{Observation}

\newtheorem{definition}[thm]{Definition}

\newtheorem*{theorem}{Theorem}

\begin{document}
\title[Anticyclotomic $p$-adic $L$-functions]{Anticyclotomic $p$-adic $L$-functions for elliptic curves at
  some additive reduction primes}

\author{Daniel Kohen}

\address{Departamento de Matem\'atica, Facultad de Ciencias Exactas y Naturales, 
Universidad de Buenos Aires and IMAS, CONICET, Argentina}
\email{dkohen@dm.uba.ar}
\thanks{DK was partially supported by a CONICET doctoral fellowship}

\author{Ariel Pacetti} \address{FAMAF-CIEM, Universidad Nacional de
  C\'ordoba. C.P:5000, C\'ordoba, Argentina.}
\email{apacetti@famaf.unc.edu.ar} \thanks{AP was partially supported by PIP 2014-2016 11220130100073} \keywords{}

\begin{abstract}
  Let $E$ be a rational elliptic curve and let $p$ be an odd prime of
additive reduction. Let $K$ be an imaginary quadratic field
  and fix a positive integer $c$ prime to the conductor of $E$.  The main 
goal of the present article is to define an anticyclotomic $p$-adic 
$L$-function $\L$ attached to $E/K$ when $E/\QQ_p$
  attains semistable reduction over an abelian extension.
  We prove that $\L$ satisfies the expected interpolation properties;
  namely, we show that if $\chi$ is an anticyclotomic character of conductor 
$cp^n$
  then $\chi(\L)$ is equal (up to explicit constants) to 
  $L(E,\chi,1)$ or $L'(E,\chi,1)$.
\end{abstract}				
\maketitle
		
\section*{Introduction}
The theorem of Mordell states that the rank of a rational elliptic curve $E$ is 
finite. It is a hard and interesting problem
to determine it and, furthermore, to compute a set of generators
for $E(\QQ).$ By Weil's generalization of Mordell's result, the rank is still 
finite over number fields $L$. Although the rank cannot be bounded
over arbitrary algebraic extensions, sometimes this is still
the case, for example,  Mazur (\cite{Mazur})  proved
that if $\Sigma$ is a finite set of primes then $E(\QQ_\Sigma^{ab})$ is 
finitely generated, where
$\QQ_\Sigma^{ab}$ denotes the maximal abelian extension of $\QQ$
unramified outside $\Sigma$. 

The techniques used to bound the rank of $E/L$ involve a detailed
analysis of the Selmer group. If $L$ is the $\ZZ_{p}$-extension of $\QQ$, that 
is, a Galois extension with Galois group 
 isomorphic to $\ZZ_p$, a deep
conjecture of Iwasawa relates the dual of the $p$-primary part of this Selmer 
group to a $p$-adic
analytic object called the cyclotomic $p$-adic $L$-function of $E$. The
study and definition of such $p$-adic $L$-function was considered by many 
authors
(\cite{Mazur-SD,Amice-Velu,Vishik,MTT}).

A natural variation of the problem is to start with a base field $K$,
and study the rank of $E$ over a $\ZZ_p$-extension $L/K$. When $K$ is
an imaginary quadratic field, any such extension is contained in
the compositum of the $\ZZ_{p}$-cyclotomic extension (lying inside the extension obtained by adjoining the
$p^{n}$-th roots of
unity for every $n \in \NN$) and the so called $\ZZ_p$-anticyclotomic extension (a generalized
dihedral extension of $\QQ$). These two extensions are the only
ones that are Galois over $\QQ$. A good reason to study the anticyclotomic $\ZZ_p$-
extension
is that if $\chi$ is an anticyclotomic character then the $L$-function
$L(E,\chi,s)$ satisfies a functional equation and its central value
holds important arithmetic information. The $p$-adic $L$-function $\L$ is a
$p$-adic analytic function that should encode the central values $L(E,\chi,1)$
(or its derivative $L'(E,\chi,1)$) for finite order anticyclotomic characters
$\chi$.

The study of the rank behavior over the anticyclotomic extension, and
the generalization of Iwasawa's conjecture to this setting was
pioneered by Bertolini and Darmon (see for example the breakthrough
papers \cite{B-D1,B-D2}), where they prove (among many important
properties) one divisibility of the anticyclotomic Iwasawa's main
conjecture. The strategy in this setting is to construct special
geometric objects (CM points) arising from orders in the imaginary
quadratic field $K$ satisfying compatibility relations. More
precisely, let $N$ be the conductor of $E$, let $c$ be a positive
integer prime to $N$, and let $G_{n}:=\Gal(H_{n}/K)$, where $H_{n}$
denotes the ring class field of conductor $cp^n$. The special points
allow to construct a $p$-adic measure on the Galois group
$G_{\infty}:= \varprojlim G_n$ (such measure is naturally defined in
the characteristic functions of the sets $G_n$ for each $n$ and
extended by continuity to locally constant $p$-adic functions). To
ensure the additive property of the measure a suitable normalization
of the geometric points is needed. In \cite{MTT} a normalization is
presented using the action of the $U_p$ operator and its
eigenvalues. This imposes an extra condition at $p$, namely the curve
must be semistable ordinary at $p$ (the supersingular case was
considered by Pollack \cite{Pollack} and Darmon-Iovita
\cite{Darmon-Iovita} in the cyclotomic and anticyclotomic setting
respectively). 

Perrin-Riou (\cite{PR}) gave a very general construction of the
$p$-adic $L$-function once a local condition at $p$ is imposed (see
Theorem 16.4 of \cite{Kato}) from the data of an Euler system and Kato
constructed such an Euler system for modular forms. The local
condition at $p$ for the $p$-adic $L$-function can be understood as
choosing a ``canonical'' direction to project such cohomological
classes. In the multiplicative reduction case one can take the
submodule given by the line fixed by inertia, while in the good
ordinary reduction case the natural choice is to take the same
submodule of the $p$-stabilized form attached to $E$. The problem is
that when $p^2 \mid N$, there is no canonical choice!  This
obstruction continues to hold in the anticyclotomic scenario
considered by Bertolini-Darmon.  Nevertheless, even when $E$ has
additive reduction at $p$, there are some instances where a natural
normalization can be taken, namely when $E/\QQ_p$ attains
\emph{semistable reduction over an abelian extension} (SRAE) of
$\QQ_p$.  This approach was carried over by Delbourgo \cite{Delbourgo}
in the cyclotomic case and the main contribution of this article is to
make an analogous construction in the anticyclotomic scenario. 

To keep the statement as simple as possible, we state our main result with
some extra hypotheses: let $E$ be an elliptic curve of conductor $N$,
with $p$ a SRAE prime of additive reduction which is not a quadratic
twist of an elliptic curve semistable at $p$, and let $\chi$ be a family of
anticyclotomic characters of conductor $cp^n$, with $n \ge 1$. The sign of the functional equation of $L(E,\chi,s)$ is constant on this
family; suppose it equals $+1$.

\begin{theorem}
  With the above hypotheses, there exists an antyciclotomic $p$-adic
  $L$-function $\L \in \ZZ_p[{G}_\infty]$ which satisfies the
  following interpolation properties:
\[
 \chi(\mathscr{L})=\frac{p^n}{{\alpha}^{2n}} \cdot 
\frac{L(1,E,\chi)}{\Omega'_E} \cdot 
\frac{u^2_{n}\sqrt{D}c}{2^{-\#{\Sigma_D}}},
  \]
  where $\Omega'_E$ is a period, $\Sigma_D$ is the set of places
  dividing both $N/p^2$ and the discriminant $D$ of $K$, $u_{n}$ is half the number of units of the order of conductor $cp^n$
  and $\alpha$ is a $p$-adic unit which
   depends only on $E$.
\end{theorem}

We prove a stronger result valid for any elliptic curve $E$ for which
$p$ is a SRAE prime and including a slightly more general class of
characters $\chi$. Furthermore, when the functional equation sign in
the family equals $-1$, we have a similar theorem, but replacing
$L(E,\chi,1)$ with the special values of the derivative
$L'(E,\chi,1)$.  See Theorems ~\ref{thm:padilinterpolation} and
\ref{thm:interp} for the precise statements.

\medskip

Our strategy is as follows: the modularity of rational elliptic
curves (due to Wiles et al. \cite{Wiles,BCDT}) implies that there
exists an automorphic representation $\pi_E$ of $\text{GL}_2(\AA_\QQ)$
with trivial central character whose $L$-series coincides with that of
$E$. The SRAE at $p$ hypothesis (for $p \ge 3$) is equivalent to
$\pi_E$ being a Steinberg representation or a ramified principal
series at $p$.  Then there exists a Dirichlet
character $\psi$ and an automorphic form $\pi_g$ whose level has
valuation at most $1$ at $p$ (with non-trivial Nebentypus in general)
such that $\pi_g \otimes \psi = \pi_E$. Following the general
philosophy, the restriction of the $p$-adic Galois representation
attached to $\pi_g$ (by Deligne) to the local Galois group
$\Gal(\overline{\QQ_p}/\QQ_p)$ does have a stable line (hence a
natural submodule).

Concretely, the form $\pi_g$ has an abelian surface $A_g$
attached to it (of $\GL_2$-type, whose endomorphism ring
$\End_\QQ(A_g)\otimes \QQ$ isomorphic to $\QQ$, $\QQ(\sqrt{-1})$ or
$\QQ(\sqrt{-3})$) \cite[Section 2.1]{Kohen} where we make a
\emph{classical} construction of CM points on $A_g$ (as we did in
\cite{KP} for constructing Heegner points for SRAE primes ramifying in
$K$) and use them to define the $p$-adic $L$-function of $E$. Clearly,
the $p$-adic $L$-function of $E$ and that of $A_g$ should be related
by a ``shift'' on the analytic functions space (corresponding to the
twist by $\psi$). The main novelty of the present article is that the
special points used to construct the $p$-adic $L$-function of $E$ are
in $A_g$ (not in $E$); still their existence and properties are enough
to define the $p$-adic $L$-function.

The second goal of the article is to prove the interpolation
properties of the $p$-adic $L$-function.  In order to prove it we make
heavy use of the fact that the CM points used to define the $p$-adic
$L$-function have heights related to central values, as proved by
Waldspurger and by Gross-Zagier (in our setting the explicit formulas
are due to Cai-Shu-Tian \cite{Cai-Shu-Tian}). Note that special values
$L(E,\chi,1)$ are related to $L(A_g,\psi\chi,1)$, which justifies
working with $A_g$ instead of $E$. The results we obtained are similar
in spirit to the ones by Chida-Hsieh \cite{Chida} and Van Order
\cite{Vanorder} but they only consider the case where the reduction at
$p$ is semistable.

In addition, Disegni on \cite{Disegni} deals with a much more general situation 
but under the hypothesis that the prime $p$ splits in $K$ (in that case our 
result can be obtained  by plugging the corresponding test vector in his 
formula). We want to stress 
that we do not make any assumptions on the
factorization of $p$ in $K$: it could be split, inert or ramified. The
ramified case is of special interest as it is widely overlooked in the
literature (in the semistable case  see the very recent preprint of Longo-Pati \cite{LP}).
In a sequel article, we will use the present construction to prove one divisibility of Iwasawa's main conjecture.

To ease the exposition, we assume that the level of $\pi_g$ is
divisible by $p$ (i.e. $E$ is not the quadratic twist of an elliptic
curve with good reduction at $p$). In the last section we explain the
changes needed to handle this case.

The method described in the present article can be used to handle the
case of newforms in $S_k(\Gamma_0(N))$, for arbitrary weights $k$,
whose level $N$ is exactly divisible by $p^2$ with the conditions:
\begin{enumerate}
\item The local component at $p$ is not supercuspidal,
\item The $L$-series $L(f,\chi,s)$ has functional equation sign $+1$
  (so as to work with definite quaternion algebras).
\end{enumerate}
The techniques are developed in \cite{Chida} in the semistable case,
and our technique can be applied with the natural modifications.

\smallskip

\textbf{Acknowledgments:} We would like to thank David Loeffler for
suggesting the present problem as an application of the results in \cite{KP}. We would also like to thank the referee whose comments helped to considerably improve the exposition of the paper.

\section*{Setting and notation}
We fix the following hypotheses and notation throughout the article:

\begin{itemize}
\item Let $p$ be a fixed odd prime number.
\item Let $E/\QQ$ be an elliptic curve of conductor $N$ with SRAE at $p$. Let $\pi_E$ be the automorphic representation of $\text{GL}_2(\AA_\QQ)$ attached to $E$.
\item As explained in the introduction, $\pi_g$ denotes an automorphic
  representation with $v_p(\text{cond}(\pi_g)) \le 1$, and $\psi$
  denotes a character of conductor $p$ such that
  $\pi_g \otimes \psi=\pi_E$. We assume in all sections but the last
  one that $v_p(\text{cond}(\pi_g))=1$.
\item Let $K$ be an imaginary quadratic field and let $\eta$ be the
quadratic Hecke character in correspondence with $K$ via class field
theory. 

\item Let $c$ be a positive integer relatively prime to $N$ (in particular $p \nmid c$).
\item  For $d \in \NN$, let
$\On_{d}:=\ZZ+ d \On_{K}$ be the order in $K$ of conductor $d$. 
\item Let $H_n$ be the ring class field of
conductor $cp^n$ and let $\widetilde{H}_{n}=H_{n}(\overline{\QQ}^{\text{ker}(\psi)})$. We define the Galois groups 
$G_{n}:=\Gal(H_{n}/K)$ and $\widetilde{G}_n:=\Gal(\widetilde{H}_{n}/K)$ and their 
respective limits
$G_{\infty}:= \varprojlim G_n$, $\widetilde{G}_{\infty}:= \varprojlim \widetilde{G}_n$.
\item $\chi$ will denote a finite order anticyclotomic character  of $K$, 
i.e. $\chi: K^{\times} \backslash \mathbb{A}^{\times}_{K} \rightarrow
\CC^{\times}$ denotes a finite order Hecke character whose restriction to
$ \mathbb{A}^{\times}_{\QQ}$ is trivial.
\item For $\Sigma$  a finite set of places and 
$L^{(\Sigma)}(E,\chi,s)$ denotes
the classical $L$-series, with the factors at primes in $\Sigma$
removed.
\item  $L^{\epsilon}(E,\chi,s)$ denotes the L-series 
for
$\epsilon=0$ and its derivative for $\epsilon=1$.
\item If $M$ is a $\ZZ$-module, we denote by $M_p$ the extension of
  scalars to $\ZZ_p$, namely $M_p = M \otimes_\ZZ \ZZ_p$.
\item $\widehat{\ZZ}$ denotes the profinite integers, namely
  $\widehat{\ZZ}:= \varprojlim \ZZ/N\ZZ = \prod_{p} \ZZ_{p}$. If $M$
  is a $\ZZ$-module we write $\widehat{M}:= M \otimes \widehat{\ZZ}$.
\item $B$ will denote a rational quaternion algebra, and
  $\hat{B} = B \otimes_\QQ \hat{\QQ}$.
\item $R$ will denote an order in $B$, and consistently $\hat{R}= R \otimes_\ZZ \hat{\ZZ}$.
\item If $B$ is a rational quaternion algebra split at $p$, and
  $M \in \M_2(\QQ_p) \cong B_p$,  $M^{(p)}$ denotes the element
  in $\widehat{B}$ whose $p$-th entry equals $M$ and the others equal
  $1$.
\end{itemize}

\section{Quaternion algebras and CM points}\label{section:Quaternion}
From now on, we will let $E$ be a fixed elliptic curve of conductor
$N$ with SRAE at $p$, and $\pi_g$ the automorphic representation with
$v_p(\cond(\pi_g)) = 1$. Let $K$ be an imaginary quadratic field,
corresponding via class field theory to a quadratic character $\eta$.

Let $\chi$ be an anticyclotomic character of $K$ whose conductor
divides $cp^n$; this corresponds to a character of $\Gal(K^{ab}/K)$
factoring through $G_n$.  The anticyclotomic assumption implies that
the twisted $L$-function $L(\pi_E,\chi,s)$ satisfies a functional
equation
\[
L(\pi_E,\chi,s)=\varepsilon(\pi_E,\chi,s)L(\pi_E,\chi,2-s),
\]
where $\varepsilon(\pi_E,\chi,s)$ is the so called \emph{epsilon
  factor} (for definitions and facts regarding such $L$-series,
consult \cite[Chapter IV]{MR0562503}). The global root number
$\varepsilon(\pi_E,\chi,1)$ can be computed as the product of local
root numbers $\varepsilon({\pi_E}_v,{\chi}_v,1)$ each of them being
$\pm 1$ (see \cite{Deligne1973}).  Consider the set
\[S(\chi):= \left\lbrace v: \, \, \varepsilon({\pi_E}_{v},{\chi}_v,1)
  \neq \chi_{v}(-1)\eta_{v}(-1) \right\rbrace. \]
By Theorem 1.3 of \cite{MR3237437},
$\varepsilon({\pi_E},{\chi},1)= (-1)^{\# S(\chi)}$ and thus the parity
of the size of $S(\chi)$ determines the parity of the order of
vanishing of $L(\pi_E,\chi,s)$ at $s=1$.

\begin{prop} The set $S(\chi)$ satisfies the following properties:
  \leavevmode \begin{enumerate}
 \item The archimedean prime $\infty$ belongs to $S(\chi)$.
 \item If $v \neq p$ is a non-archimedean prime, then the condition
   ``$v \in S(\chi)$'' depends only on $K$, i.e. is independent of
   $\chi$.
 \item The prime $p$ does not belong to $S(\chi)$ if either
\begin{itemize}
\item The local Weil-Deligne representation of $E$ at $p$ is a principal series.
\item The prime $p$ splits in $K$.
\item The prime $p$ is inert in $K$ and $\chi_p$ is not equal to the quadratic 
character modulo $p$.
\item The prime $p$ is ramified in $K$ and $\chi_p$ is not trivial.
\end{itemize}
\end{enumerate}
\label{prop:Schi}
\end{prop}
\begin{proof}
  The first statement follows from \cite[Proposition 6.5]{MR970123},
  while the second one follows from the assumption that $\gcd(c,N)=1$.
  Regarding the last one, the assumption on $p$ being a SRAE
  prime implies that the local representation
  of $E$ at $p$ is either a twist of Steinberg or a principal
  series. The result then follows from \cite[Propositions 1.6 and
  1.7]{Tunnell}.
\end{proof}

In particular, for all but finitely many characters $\chi$ of
conductor $cp^n$ (with $p \nmid c$), the set $S(\chi)$ is
constant. Let $S$ denote such generic common set. Let
$\epsilon \in \{0,1\}$ be such that $\epsilon \equiv \#S \pmod 2$. By
$L^{\epsilon}(\pi_E,\chi,s)$ we denote the $L$-series
$L(\pi_E,\chi,1)$ if $\epsilon = 0$ and its derivative
$L'(\pi_E,\chi,s)$ if $\epsilon =1$. Our main goal is to interpolate
the special values $L^\epsilon(\pi_E,\chi,1)$.

\medskip
To relate central values of $\pi_E$ to those of $\pig$, let
\begin{equation*}
\widetilde{\chi}:= \chi \cdot(\psi\circ
Nm_{\mathbb{A}^{\times}_{K}/\mathbb{A}^{\times}_{\QQ}}): K^{\times}
\backslash \mathbb{A}^{\times}_{K} \rightarrow \CC^{\times}.  
\end{equation*}
Since $\widetilde{\chi} \mid_{\mathbb{A}^{\times}_{\QQ}}=\psi^{2}$,
$L(\pig,\widetilde{\chi},s)$ is self dual and clearly
$L(\pi_E,\chi,s)=L^{(\{p\})}(\pig,\widetilde{\chi},s)$.

\medskip

\begin{definition}
  The character $\chi$ is \emph{good} if the conductor of $\tchi$ is
  divisible by $p$.
\end{definition}
If $\tchi$ has conductor $cp^n$ (with $p \nmid c$), we will see that
the central value $L^\epsilon(\pi_g,\tchi,1)$ is related to the height
of a linear combination of CM points of conductor $cp^n$. Varying the
character's conductor, involves constructing CM points of different
conductors and good characters correspond to good CM points in the
sense of Cornut-Vatsal \cite[Definition 1.6]{Cornut-Vatsal}, that will
give the distribution relations needed to define a $p$-adic
measure. Note that Proposition~\ref{prop:Schi} implies that if $\chi$
is good, $p \not \in S(\chi)$. From now on, we will only work with
good characters.

Let $B/\QQ$ be the quaternion algebra ramified at the places of $S$ if
$\epsilon =0$ (the \emph{definite case}) and at all places of $S$ but
the infinite one if $\epsilon =1$ (the \emph{indefinite
  case}). 

\begin{lemma}
  There exists an embedding $\iota:K \mapsto B$.
\end{lemma}
\begin{proof}
  Proposition ~\ref{prop:Schi} implies that if $v$ splits in $K$, then
  $v \notin S$. The result then follows from Theoreme 3.8 of
  \cite{Vigneras}.
\end{proof}

\subsection{Quaternionic level} Given a good character $\chi$ as before we seek
for an order $R$ in $B$ and an embedding $\iota:K \to B$ with the
properties that $\pig$ transfers to an automorphic form of level $R$ and, at
the same time, $R$ contains CM points. 

\begin{definition} Let $\iota:K \to B$ be an embedding, $\id{n}$ a
positive integer divisible by $p$ with $\gcd(\id{n},N/p^2)=1$ and
$R \subset B$ be an order. We say that $R$ is \emph{admissible} for
$(\pi_g,\id{n},\iota)$ if $R_{p}$ is an Eichler order of level
$p\ZZ_p$ and $\iota$ is an optimal embedding of $\On_{\id{n}}$ into
$R$, that is, $\iota(K) \cap R= \iota(\On_{\id{n}})$.
\end{definition} 
\begin{remark}
  If $\tchi$ is a character whose conductor is divisible by $p$, then
  our admissibility condition for $(\pi_g,\cond(\tchi),\iota)$ implies
  admissibility in the sense of \cite[Definition 1.3]{Cai-Shu-Tian}.
\end{remark}

Given an embedding $\iota$ and a good character $\tchi$, there always
exists an admissible order $R$ for $(\pi_g,\cond(\tchi),\iota)$ by \cite[Propositions 3.2,
3.4]{MR970123}, \cite[Lemma 3.2]{Cai-Shu-Tian} and the local-global
principle. Still, for explicit computations, it is useful to choose $R$
such that its completion $R_p$ matches the standard Eichler
order. This can be achieved allowing to change the embedding to an
equivalent one.

\begin{lemma}\label{lemma:iotachoice}
  Let $c$ be a positive integer prime to $p$. Then, there exists an
  embedding $\iota:K \to B$ and an order $\R\subset B$ which is
  admissible for $(\pig,cp,\iota)$ whose completion
  at $p$ is the standard Eichler order of level $p\ZZ_{p}$.
\end{lemma}
\begin{proof} Let $R$ be any admissible order for
  $(\pi_g,cp,\iota)$. Locally, $R_p$ is conjugate to the standard
  Eichler order, but by weak approximation we can find a global
  element that sends this order to the standard one. Conjugating both
  $\iota$ and the order the result follows.
\end{proof}
Fix once and for all $\R$ and $\iota$ as in the lemma.  For $n \ge 1$,
let
$\g_{n}:=\left( \begin{smallmatrix} p^{n-1}&0\\
    0&1 \end{smallmatrix} \right)^{(p)} \in {\widehat{B}}^{\times}$
(see the notations section).
\begin{lemma}
  Let $n \ge 1$ be a positive integer. The order
  $\R_n:=\g_{n} \widehat{\R} \g_n^{-1} \cap B$ is admissible for
  $(\pig,cp^n,\iota)$.
\label{lemma:admissible}
\end{lemma}
\begin{proof} Let $\omega' \in K$ be such that
  $\On_{c}= \ZZ+ \omega' \ZZ$. Then, $\On_{cp}= \ZZ+ \omega \ZZ$, where $\omega:= p \omega'$.
  Since the order $\R$
  is admissible for $(\pig,cp,\iota)$, the $p$-th component of
  the image of $\omega$ under $\iota$ is a matrix
  $\left( \begin{smallmatrix} a&b\\ c&d \end{smallmatrix} \right) \in M_{2}(\ZZ_p)$
  such that $p$ divides $a,c,d$ and does not divide $b$. Moreover,
  $\On_{cp^n}= \ZZ+ p^{n-1}\omega \ZZ$ and
 
 \[ \left( \begin{smallmatrix} p^{n-1}&0\\ 0&1 \end{smallmatrix} \right)^{-1} p^{n-1}\left( \begin{smallmatrix} a&b\\ c&d \end{smallmatrix} \right) \left( \begin{smallmatrix} p^{n-1}&0\\ 0&1 \end{smallmatrix} \right)  \]
  is a matrix whose entries are $p$-integers and its $(1,2)$ entry is not divisible by $p$. This shows that $\iota$ is an optimal
  embedding of $\On_{cp^n}$ into the order $\R_n$ as stated. 
\end{proof}
Let $U$ be an open compact subgroup of ${\widehat{B}}^{\times}$.
If $B$ is definite, let
\[ X_{U}:= B^{\times}  \backslash 
 {\widehat{B}}^{\times}  /U, 
\]
where $U$ acts on ${\widehat{B}}^{\times}$ by right multiplication and $B^{\times}$ acts on ${\widehat{B}}^{\times}$
by left multiplication. If $B$ is indefinite, let
\[ X_{U}:= B^{\times}  \backslash 
(\CC-\RR) \times {\widehat{B}}^{\times}  /U,
\]
where $U$ acts trivially on $\CC-\RR$ while
$B^{\times}$ acts on $\CC-\RR$ by M\"obius transformations under the identification
$B_{\infty} \cong \text{M}_{2}(\RR)$.
If $R$ is an order in $B$ we write $X_R:=X_{\widehat{R}^{\times}}$.
\begin{remark}
  In the definite case, the curves $X_U$ are $0$-dimensional (i.e. are
  finite sets) while in the indefinite case, they have dimension
  $1$. In the latter case, we denote by $J_U$ its Jacobian
  variety. 
\end{remark}
Let $X:= \varinjlim_U X_U$ and $J:= \varinjlim_U J_U$, where the limit
is induced by the natural projection arising from the inclusion of
level structures.  Since $\R_{p} \subset B_{p} \cong M_{2}(\QQ_p)$
we can regard $\psi$ as a character on $\widehat{\R}^{\times}$ by the
reduction modulo $p$ of the $(2,2)$-entry of $\R_{p}$. Recall that in
the introduction we defined the abelian variety $A_{g}/\QQ$ associated
to $\pi_g$ (see Section 6.6 of \cite{Diamond} for more details, in
particular Theorem 6.6.6)

\begin{thm}[Jacquet-Langlands] With the same notations as before, there is an
  automorphic transfer of the form $\pig$ to the algebra $\widehat{B}^\times$. Furthermore,
\label{thm:JL}
\leavevmode \begin{enumerate}
\item if $B$ is definite, there exists an automorphic form
  $\JLg:B^\times \backslash \widehat{B}^\times \to \CC$, such that
  \[
    r \cdot \JLg(x):= \JLg(xr) = \psi^{-2}(r) \JLg(x) \text{ for all } r \in \widehat{\R}^\times.
    \]
  \item If $B$ is indefinite, there exists
    $\JLg \in \text{Hom}(J,A_{g}) \otimes_{\ZZ} \QQ$ such that
     \[
    r \cdot \JLg = \psi^{-2}(r) \JLg \text{ for all } r \in \widehat{\R}^\times,
    \]
    where $\text{Hom}(J,A_{g}) \otimes_{\ZZ} \QQ$ is  endowed
    with the right Hecke action of $\widehat{B}^\times$ inherited from $X$.
    
 \end{enumerate}
 Moreover, if all primes $q \neq p$ such that
 $q^2 \mid N$ are unramified in $K$, the form
 $\JLg$ is unique up to a constant.
\end{thm}
\begin{proof}
  The existence of the form 
  $\JLg$ and its uniqueness follow from \cite[Proposition
  2.6]{Gross-Prasad} and  \cite[Propositions 3.7 and 3.8]{Cai-Shu-Tian}) combined with the Jacquet-Langlands philosophy.
 
\end{proof}

\subsection{CM points}
The embedding $\iota: K \hookrightarrow B$ induces an embedding
$\widehat{\iota}: {\widehat{K}}^{\times} \hookrightarrow {\widehat{B}}^{\times}$.
If $B$ is indefinite, let $z_0$ be the unique fixed point on the upper half plane under the
action of $K^{\times}$. Define
\begin{equation*}
  P=[(z_0,1)] \in X,
\end{equation*}
where if $B$ is definite, abusing notation, the point $[(z_0,b)]$
denotes the class of $b \in {\widehat{B}}^{\times}$ in $X$. Let
  \[
\U:= \left\lbrace \left( x_{\ell} \right)_{\ell} \in \widehat{\R}^{\times} : 
x_{p} \equiv \left( \begin{smallmatrix} *&*\\ 0&a \end{smallmatrix} 
\right) \bmod{p} \text{ with }\psi^{2}(a)=1 \right\rbrace.
\]

Note that from Theorem \ref{thm:JL} we immediately obtain that the
form $g_B$ is invariant under the action of $\U$ (so we can think of
$g_B$ as a form with ``trivial Nebentypus'' with respect to the level
$\U$). The inclusion $\U \subseteq \widehat{\R}^{\times}$ induces a
quotient map $\beta: X_{\U} \rightarrow X_{\R}$.

\begin{definition} \leavevmode
\begin{itemize}
 \item A \emph{CM point of conductor $cp^n$} on $X_{\R}$ is a pair
  $[z_0,b] \in X_\R$, where $b \in {\widehat{B}}^{\times}$ is such that
  $\iota$ is an optimal embedding of $\O_{cp^n}$ into $b \widehat{\R} b^{-1} \cap B$.
\item For $n \ge 1$, the \emph{CM points of conductor $cp^n$} on
  $X_{\U}$ are the preimages under $\beta$ of CM points of conductor
  $cp^n$ on $X_{\R}$.
 \end{itemize}
\end{definition}
 Let
\begin{equation*}
  \label{eq:CMpoints}
  \z_{n}:=\g_n \cdot P=\left[\left(z_0,\left( \begin{smallmatrix} p^{n-1}&0\\ 0&1 \end{smallmatrix} \right)^{(p)}\right)\right]  \in X_\R .
\end{equation*}
\begin{prop}
  The points $\z_n$ (for $n \ge 1$) are CM points of conductor $cp^n$
  on $X_{\R}$. In particular their preimages under $\beta$ are CM
  points on $X_{\U}$.
\end{prop}
\begin{proof}
This follows immediately from Lemma~\ref{lemma:admissible}.
\end{proof}

There is a natural action of $\Gal(K^{ab}/K) \cong  K^\times \backslash \widehat{K}^\times$ on CM points given by
  \[ \id{a}\cdot[(z_0,b)] := [(z_0,\widehat{\iota}(\id{a})b)]  .\]
 In the indefinite case the Galois action is the natural one on algebraic points. However, in the definite scenario we do not have a similar interpretation.

Consider the operator $U_p$ whose  action on both $\Div(X_{\R})$ and $\Div(X_{\U})$ is given by 
\begin{equation*}
  \label{eq:Upoperator}
U_{p}([(z_0,b)])= \sum_{i=0}^{p-1} [(z_0,b\left( \begin{smallmatrix} p&i\\ 0&1 
\end{smallmatrix} \right)^{(p)})].
\end{equation*}
 The interplay between the Galois action and the $U_{p}$ action on CM
 points is as follows.
\begin{prop} Let $n \ge 1$.
  \label{prop:galois} 
  \leavevmode
\begin{enumerate}
\item If $B$ is definite $\z_{n} \in H^{0}( \Gal_{\widetilde{H}_{n}}, X_{\U})$ and if $B$ is indefinite $\z_n \in X_{\U}(\widetilde{H}_{n})$.
\item $\sum_{\sigma \in 
\Gal(\widetilde{H}_{{n+1}}/\widetilde{H}_{n})} \z_{n+1}^{\sigma} = U_{p}(\z_n)$.
\end{enumerate}
\end{prop}

 \begin{proof}
   This is essentially proved by Longo and Vigni in \cite[Propositions
   3.2, 3.3, 3.4 and Section 4.4]{Longo-Vigni}, with the remark that for
   $n \ge 1$, the second condition in their definition of Heegner
   points (Ibid Definition $3.1$) is redundant, hence it coincides with
   our definition of CM points. The only difference is that they work with a full $\Gamma_1(p)$ structure and thus their points are defined over the extension
   $H_n(\mu_p)$. But proceeding as in \cite[Proposition 2.12]{KP} we see that the points for $\U$ are defined over $\widetilde{H}_{n}$.
\end{proof}   

\section{Waldspurger and Gross-Zagier formulas}
\label{section:WGZ}
The CM points defined in the previous section are related to the
central values of $L^{\epsilon}(\pig,\tchi,s)$ via the Waldspurger
formula for $\epsilon = 0$ (the definite case) and the Gross-Zagier formula for
$\epsilon =1$ (the indefinite case). We follow the more general formulas by Yuan-Zhang-Zhang \cite{MR3237437}
and the explicit formulation given by Cai-Shu-Tian in \cite{Cai-Shu-Tian}.

Recall the choice of the ramification algebra $B$ and the ramification
set $S(\chi)$ given in Section \ref{section:Quaternion}.  By results
of Tunnel and Saito (\cite[Propositions 1.6 and 1.7]{Tunnell} and
\cite[Propositions 6.3 and  6.5]{Saito}) the space
$\text{Hom}_{K^{\times}}(\piJLg,\tchi)$ is
$1$-dimensional. 

\begin{definition}
A vector $v \in \piJLg$ is a called a
\emph{test vector} for $\widetilde{\chi}$ if
$\ell_{\widetilde{\chi}}(v) \neq 0$ for any nonzero
$\ell_{\widetilde{\chi}} \in
\text{Hom}_{K^{\times}}(\piJLg,\widetilde{\chi})$.
\end{definition}

\begin{prop}
Suppose that for every prime $q \neq p$ such that $q^2 \mid N$, $q$ is unramified in $K$.
  Let $\chi$ be a good character such that $\tchi$ is of conductor $cp^n$. Then the
  space ${\piJLg}^{{\g^{-1}_n}\U{\g_{n}}}$ is one dimensional. Moreover,
  every non-zero vector of it is a test vector for
  $\tchi$.
\end{prop}

\begin{proof}
  The follows from  \cite[Proposition
  2.6]{Gross-Prasad} and \cite[Propositions 3.7 and
  3.8]{Cai-Shu-Tian}.
\end{proof}
\begin{remark}
  In the case when there are primes $q$ ramified in $K$ such that
  $q^2 \mid N$, the local space $(\piJLg)_q^{\R_q}$ has dimension $2$,
  but there is a canonical fixed line to consider, as explained in
  \cite[Remark 2.7]{Gross-Prasad}. For the general construction, we
  take an element in such line as the test vector $\JLg$. Note that this small technical issue is not important as we will only be varying the test vectors at the prime $p$ which is different from any such $q$.
\label{remark:unramified}
\end{remark}
For $n \ge 1$, consider the vector $\phi_{n}:=\g_n \cdot \JLg \in \piJLg$. 
\begin{lemma} The vector $\phi_n$ is a non-zero test vector for
  $\tchi$.  The complex conjugate of $\phi_{n}$ viewed as an
  element of $\piJLg^{\vee}$ is a non-zero test vector for
  ${\widetilde{\chi}}^{-1}$.
\end{lemma}

\begin{proof}
The statement follows from the fact that $\JLg \in \piJLg$ is invariant under the action of $\U$.
\end{proof}

Let $\Z$ be $\ZZ$ in the definite case and
$A_g(\overline{\QQ})$ in the indefinite one.  The projection of $P=[(z_0,1)]$ to the 
$\tchi$-isotypical component in $\Z$ is
given by
\[
  P_{\widetilde{\chi}}(\phi_{n}):= \sum_{\sigma \in \Gal(\widetilde{H}_{n}/K)} 
  \phi_{n}(P^\sigma)\widetilde{\chi}(\sigma)=\sum_{\sigma \in  \Gal(\widetilde{H}_{n}/K)} \JLg(\z_{n}^{\sigma}) \widetilde{\chi}(\sigma) \in (\Z \otimes \CC)^{\widetilde{\chi}}.
\]
Let $\Sigma_D$ be the set of places dividing both $N/p^2$ and the discriminant 
$D$ of $K$ and let $u_{n}:=\#{\On_{cp^n}}^{\times}/2$. Let $\left\langle -,- \right\rangle$ denote the natural pairing in
$\Z$, i.e.  multiplication in the definite case and the N\'eron-Tate
pairing in the indefinite one. We are now able to state the explicit version of Gross-Zagier and Waldspurger formulas.

\begin{thm}
  Let $\chi$ a good character, and let $cp^n$ be the conductor of
  $\widetilde{\chi}$. Then
\[L^{\epsilon,\{p\}}(1,\pi_g,\tchi):=\frac{2^{-\#{\Sigma_D}}8\pi^{2}
\left\langle g,\overline{g} 
\right\rangle_{U_{0}(N/p)}}{{u^2_{n}}\sqrt{D}cp^n}\frac{\left\langle 
P_{\widetilde{\chi}}(\phi_{n}),P_ { \widetilde { \chi
} ^ { -1 } }
(\overline{\phi_{n}})\right\rangle}{\left\langle \phi_{n}, \overline{\phi_{n}} 
\right\rangle_{{\delta^{-1}_n}\U{\delta_{n}}}}.
\]

\label{thm:GZ}
\end{thm}
\begin{proof}
  See \cite{Cai-Shu-Tian} Theorem 1.8 for $\epsilon=0$ and Theorem 1.5
  for $\epsilon=1$.
\end{proof}
\section{Anticyclotomic p-adic L-function}
\label{section:anti}
The $p$-adic $L$-function is a functional on locally constant
functions attached to a $p$-adic measure $\mu_E$, i.e. if $h$ is a
locally constant function, $\L_p(h)=\int h d\mu_E$. We will construct it using the $CM$
points we defined. Once the $p$-adic
$L$-function is defined, we will use the results of the previous
sections to relate its values at characters $\chi$ with the 
values  $L^\epsilon(\pi_E,\chi,1)$.

A crucial hypothesis in the classical constructions is that $\pi_E$
has an eigenvalue for the $U_{p}$ operator with \emph{small}
slope. Since $E$ has additive reduction at $p$ its unique eigenvalue
for $U_p$ is $0$. However, under our working assumptions $E$ has SRAE
at $p$ so we can bypass this considering the abelian variety $A_g$.
Let $\alpha$ be the eigenvalue of the $U_p$-operator acting on
$\JLg$. If $f$ is Steinberg at $p$, $\alpha= \pm 1$; otherwise the
coefficient field $M$ of $\JLg$ is a quadratic extension of $\QQ$
(either $\QQ(i)$ or $\QQ(\sqrt{-3})$) in which $p$ splits
(\cite[Section 2.1]{Kohen}), so there exists a prime $\id{p} \mid p$
such that $\id{p} \nmid \alpha$.  Then
$\alpha \in \mathcal{O}^{\times}_{M_{\id{p}}}\cong
{\ZZ^{\times}_{p}}$. Since the space of modular forms has an integral
basis and the modular form $g$ has eigenvalues lying in $\ZZ_p$, we
can always normalize $\JLg$ such that the images of the CM points lie
in $\Z_{p}:= \Z \otimes \ZZ_p$.
\begin{definition}
  For $n \ge 1$ the \emph{regularized} CM points on $\Z_p$ are
\[ 
{\zeta^{\sigma}_{n}}:=
\JLg(\z^{\sigma}_{n}) \alpha^{-n}.
\]
\end{definition}
\begin{prop}[Distribution relation] If $n \ge 1$, the regularized CM points satisfy the relation
 \label{prop:compatibility}
 \[\sum_{\sigma \in 
\Gal(\widetilde{H}_{{n+1}}/\widetilde{H}_{n})} \zeta^{\sigma}_{n+1}=\zeta_{n}. \]
\end{prop}
\begin{proof}
This is an immediate consequence of Proposition \ref{prop:galois}.
\end{proof}
 For $n \ge 1$ let
\begin{equation}
\widetilde{\theta}_{n}:= \sum_{\sigma \in \Gal(\widetilde{H}_{n}/K)} {\zeta^{\sigma}_{n}} \sigma \in 
\Z_{p}[\widetilde{G}_{n}].  
\label{thetantildedef}
\end{equation}
The compatibility relation allows to attach a $p$-adic measure to
$\JLg$, since it gives a well defined element
\[\widetilde{\theta}:=\varprojlim_{n} \widetilde{\theta}_{n} \in
\Z_{p}[[\widetilde{G}_{\infty}]]. \]
Its twisted version (that will give rise to the $p$-adic $L$-function of $\pi_E$) is defined by 
\begin{equation}
{\theta}_{n}:= \sum_{\sigma \in \Gal(\widetilde{H}_{n}/K)} \psi(\sigma){\zeta^{\sigma}_{n}}\sigma \in 
\Z_{p}[\widetilde{G}_{n}],  
\label{thetandef}
\end{equation}
where by class field theory, we can think of $\psi$ as a character of
$\Gal(\overline{\QQ}/\QQ)$ factoring through $\widetilde{H}_{n}$. It is clear from the definition
that $\psi$ is compatible with the natural map
$\widetilde{G}_{n+1} \rightarrow \widetilde{G}_n$ hence we also get a well
defined object
\[
{\theta}:=\varprojlim_{n} {\theta}_{n} \in
\Z_{p}[[\widetilde{G}_{\infty}]]. 
\]
Let $\mu_{\JLg,\alpha}$ (respectively $\mu_E$) denote the measure on
$\widetilde{G}_{\infty}$ attached to $\widetilde{\theta}$
(resp. $\theta$). Note that for $\sigma \in \widetilde{G}_\infty$, the
two measures satisfy that
$$\psi(\sigma)\mu_{\JLg,\alpha}(\sigma)=\mu_E(\sigma).$$
If $\chi$ is a good character such that $\tchi$ is
of conductor $cp^n$, then
\[
\widetilde{\chi}(\widetilde{\theta})=\int_{\widetilde{G}_{\infty}}
\widetilde{\chi}(g)d\mu_{\JLg,\alpha}(g).
\]
The character $\widetilde{\chi}$ factors through $\Gal(\widetilde{H}_{n}/K)$ so the integral equals the finite sum
\[ 
\widetilde{\chi}(\widetilde{\theta})=\sum_{\sigma \in \Gal(\widetilde{H}_{n}/K)} 
\widetilde{\chi}(\sigma){\zeta^{\sigma}_{n}}.  
\]
Looking at the definitions of $\tchi$ and
 $\widetilde{\theta}$ it is clear that
$ \widetilde{\chi}(\widetilde{\theta})= \chi(\theta)$. In particular, a
similar formula holds for $\chi(\theta)$.

One should think of $\theta$ as the square root of the $p$-adic
$L$-function. More precisely, let $*$ be the involution sending
$\sigma$ to $\sigma^{-1}$ and let
 \[ 
\L_{n}:= \theta_{n} \otimes \theta^{*}_{n} \in 
 (\Z_{p}\otimes \Z_p)[\widetilde{G}_{n}].
\]
\begin{definition}
  The $p$-adic L-function attached to $\pi_E$ is
\[
\L:=\varprojlim_{n}\mathscr{L}_{n} \in 
 (\Z_{p}\otimes \Z_p)[[\widetilde{G}_{\infty}]].
\]
\end{definition}

\begin{remark}
If we change our compatible sequence of CM points $\left\lbrace \z_n \right\rbrace_{n \ge 1}$ for another compatible sequence $\left\lbrace \z'_n \right\rbrace_{n \ge 1}$ there must exist an element $\sigma_0 \in \widetilde{G}_{\infty}$ such that for every $n \ge 1$,
$\z^{\sigma_0}_n=\z'_n$. Let  $\theta'$ be the corresponding element associated to $\left\lbrace \z'_n \right\rbrace_{n \ge 1}$. Then we have that 
\[\theta'= \psi(\sigma^{-1}) \sigma^{-1}_0 \cdot \theta. \]
Similarly, working with $\theta^{*}$ and $\theta'^{*}$ we obtain

\[\theta'^{*}= \psi(\sigma^{-1}) \sigma_0 \cdot \theta^{*}. \]

Putting these two equations together we get

\[ \L' = \psi(\sigma^{-2}_0) \L.\]

Thus $\L$ is more intrinsic that $\theta$, as it depends very mildly on the sequence of compatible CM points. The reader should compare this with \cite[Remark 1, p.12]{B-D2}.
\end{remark}
Once we fix an embedding of $\ZZ_{p}$ into
$\CC$, the pairing $\left\langle -,- \right\rangle$ induces
\[
  \left\langle -,- \right\rangle : \Z_p \otimes \Z_p \to\CC,
\]
and we let  $\L_{\CC} \in \CC[[\widetilde{G}_{\infty}]]$ be the image of
$\mathscr{L}$ under such pairing. 
\begin{prop}
\label{interpol}
Let $\chi$ be a good character such that $\widetilde{\chi}$ has conductor $cp^n$. Then
\[
  \chi(\mathscr{L}_{\CC})= \sum_{\tau_1,\tau_2 \in \widetilde{G}_n}\psi(\tau_1 \tau_2)\<\zeta_n^{\tau_1},\zeta_n^{\tau_2}>\chi(\tau_1\tau^{-1}_2).
 \]
\end{prop}
\begin{proof}
By definition, $\chi(\L_{\CC})= \chi(\<\theta_n,\theta_n^*>)$. The result follows immediately replacing $\theta_n$ and $\theta_n^*$ by their definitions (\ref{thetandef}) and (\ref{thetantildedef}).
\end{proof}
We are now ready to prove the main result of this article.

\begin{thm}[Interpolation]
There exists a constant ${\Omega'_E}$ that depends on $E$ such that for every 
 good character $\chi$ for which $\widetilde{\chi}$ has conductor $cp^n$, the following holds:
\[
 \chi(\mathscr{L}_{\CC})=\frac{p^n}{{\alpha}^{2n}} \cdot 
\frac{L^{\epsilon}(1,E,\chi)}{\Omega'_E} \cdot 
\frac{u^2_{n}\sqrt{D}c}{2^{-\#{\Sigma_D}}}.
  \]

\label{thm:padilinterpolation}
\end{thm}

\begin{proof}
At the level of modular forms we have that
$\pi_{g} \otimes \psi^2 = \pi_{\overline{g}}$. This induces the same relation under the Jacquet-Langlands transfer and we obtain that $\JLg \otimes \psi^2=\overline{g}_B$. Using the definition of $P_{\widetilde{\chi}}$ and the fact that
  $\widetilde\chi = \psi \chi$, we can write the last factor of the main
  formula of Theorem~\ref{thm:GZ} as
\[
  \<P_{\widetilde{\chi}}(\phi_{n}),P_ { \widetilde { \chi}^ { -1 }
  }(\overline{\phi_{n}})>= \sum_{\tau_{1}, \tau_2 \in \widetilde{G}_{n}}
  \psi({\tau_1} {\tau_2}) \chi(\tau_1{\tau^{-1}_2}) \left\langle
   \JLg(\z^{\tau_1}_n), \JLg(\z^{ {\tau_2} }_n)
  \right\rangle .\]
  
  Since $\JLg(\z^{\tau_1}_n) \alpha^{-n}=\zeta_n^{\tau_1}$ and
$L^{(p)}(1,\pi_{\JLg},\widetilde{\chi})=L(1,E,\chi)$ we obtain the desired result 
using Proposition \ref{interpol} and the fact that ${\left\langle \phi_{n}, 
\overline{\phi_{n}} 
\right\rangle_{{\delta^{-1}_n}\U{\delta_{n}}}}$ does not depend on $n$.
\end{proof}

\section{The good reduction twist case}\label{section:remarks}

In the case when $\pi_g$ has good reduction at $p$, i.e. $E$ is a
quadratic twist of a curve with good reduction at $p$, the previous
construction and results hold with some minor modifications. We focus
in the case when the twisted curve is ordinary at $p$, to follow the
classical construction. In the supersingular case, the same approach
works (with the additional assumption that $p$ splits in $K$), but instead of following the classical construction, one
follows the one done by Pollack in \cite{Pollack}.

The choice of level $\R=\U$ is the same, but it will be
maximal at $p$. Moreover, we can change the embedding $\iota$ in such a way that the CM point $P=\z_0=[z_0,1]$ is of conductor $c$
and $\z_n:= \g_n \cdot P$ are of conductor $cp^n$
in a similar way as we did in Lemma \ref{lemma:admissible}.  The distribution relations are the following (see for example \cite[p.433]{B-D1}):

\begin{itemize}
\item If $n \ge 1$, $\sum_{\sigma \in 
\Gal({H}_{{n+1}}/{H}_{n})} \z_{n+1}^{\sigma}= U_{p}\z_{n}-\z_{n-1}$.
\item If $n=0$,
 \[ u_{0} \cdot \sum_{\sigma \in 
\Gal({H}_{{1}}/{H}_{0})} \z_{1}^{\sigma}=\begin{cases}
(U_p-\sigma_{\mathfrak{p_{1}}}-\sigma_{\mathfrak{p_{2}}})  \z_{0} &\text{if $p$ is split in $K$},\\
(U_p-\sigma_{\mathfrak{p_{1}}}) \z_0 &\text{if $p$ is ramified in $K$} \\
U_p \z_0 &\text{if $p$ is inert in $K$},
\end{cases}   
\]
where $\sigma_{\mathfrak{p_{i}}}$ are the Frobenii of the primes above $p$ in $K$.
\end{itemize}
If $\alpha$ denotes the $p$-adic unit root of the Frobenius polynomial at $p$, the normalized CM points are defined by

\[
{\zeta_{n}}^{\sigma}:=
\begin{cases}
(\alpha \JLg(\z_{n}^{\sigma})-\JLg(\z_{n-1}^{\sigma}))\cdot \alpha^{-n-1} &\text{if }n \ge 1,\\
{u^{-1}_{0}}(1-(\sigma_{\mathfrak{p_{1}}}+\sigma_{\mathfrak{p_{2}}})\alpha^{-1}+\alpha^{-2})\JLg(\z_{0}^{\sigma})&\text{if $n=0$ and $p$ splits in $K$,}\\
{u^{-1}_{0}}(1-\sigma_{\mathfrak{p_{1}}}\alpha^{-1})\JLg(\z_{0}^{\sigma})&\text{if $n=0$ and $p$ is ramified in $K$,}\\
{u^{-1}_{0}}(1-\alpha^{-2})\JLg(\z_{0}^{\sigma})&\text{if $n=0$ and $p$ is inert in $K$.} \\
\end{cases}  
\]

The definition of the theta element and the $p$-adic $L$-function
is the same. If we take $\chi$ such that $\widetilde{\chi}$ has
conductor $cp^n$ (with $n \in \NN \cup \{0\}$), we can evaluate
$\chi(\L_{\CC})$ as we did before. 

\begin{thm}\label{thm:interp}
There exists a constant ${\Omega'_E}$ that depends on $E$ such that for every 
character $\chi$ such that $\widetilde{\chi}$ has conductor $cp^n$ the following holds:
\[
 \chi(\mathscr{L}_{\CC})= \frac{p^n}{{\alpha}^{2n}} \cdot 
\frac{{e_{p}(\tchi)}^2}{L_p(\pi_g,\tchi,1)} \cdot 
\frac{L^{\epsilon}(1,E,\chi)}{\Omega'_E} \cdot 
\frac{u^2_{n}c\sqrt{D}}{2^{-\#{\Sigma_D}}},
 \]
 
where the $p$-adic multiplier is given by

\[ e_p(\tchi):=
\begin{cases}
1 &\text{if $n \ge 1$.}\\
(1-\tchi(\sigma_{\mathfrak{p_{1}}})\alpha^{-1})(1-\tchi(\sigma_{\mathfrak{p_{2}}
} )\alpha^{-1})&\text{if $n=0$ and $p$ splits in $K$.} \\
(1-\tchi(\sigma_{\mathfrak{p_{1}}})\alpha^{-1})&\text{if $n=0$ and $p$ is 
ramified in $K$.} \\
(1-\alpha^{-2})&\text{if $n=0$ and $p$ is inert in $K$.} \\
\end{cases}.  
\]

\end{thm}
\begin{proof}
  When $n \ge 1$, if we expand the four terms in the pairing
  $\left\langle - , - \right\rangle$ we obtain that the only term that
  survives is the same as we had in the general case. The reason is
  that all the other terms involve a sum of the form
  $\sum_{\sigma} \tchi(\sigma) g_B(\z_{n-1}^{\sigma})$ which equals
  zero as the conductors of the test vector and the character do not
  agree. Finally, both Waldspurger and Gross-Zagier formulas
  (Theorem~\ref{thm:GZ}) for central values of the elliptic curve
  $E_{\pig}$ (or its derivative) need to include the $p$-th Euler
  factor at $p$ (which is trivial if $n \ge 1$ and non-vanishing in
  general) which gives the extra local factor to the formula.
\end{proof}

\bibliographystyle{alpha}
\bibliography{bibliography}
\end{document}